\pdfoutput=1
\documentclass[fontsize=12pt, paper=a4]{article}
\usepackage[a4paper, left=3cm, right=3cm, top=2cm, bottom=3cm]{geometry}

\usepackage[ansinew]{inputenc}

\usepackage{marginnote}
\usepackage{xcolor}

\newcommand{\nomgn}[1]{}

\usepackage[backend=biber,style=alphabetic]{biblatex}
\usepackage{amsmath,amsthm, amssymb}
\usepackage{breakurl}
\usepackage{tabularx}
\usepackage{enumerate}

\usepackage{enumitem}
\makeatletter
\newcommand{\mylabel}[2]{#2\def\@currentlabel{#2}\label{#1}}
\makeatother

\usepackage[colorlinks,
pdfpagelabels,
pdfstartview = FitH,
bookmarksopen = true,
bookmarksnumbered = true,
linkcolor = blue,
urlcolor = black,
plainpages = false,
hypertexnames = false,
citecolor = blue] {hyperref}

\usepackage{etoolbox}

\newcommand{\abs}[1]{\ensuremath{\left\vert#1\right\vert}}

\newcommand{\Z}{\mathbb Z}
\newcommand{\Q}{\mathbb Q}
\newcommand{\R}{\mathbb R}

\newcommand{\F}{\mathbb F}

\newcommand{\fp}{\mathfrak p}

\newcommand{\PP}{\mathcal P}

\newcommand{\NN}{\mathcal N}

\DeclareMathOperator{\li}{li}

\DeclareMathOperator{\GL}{GL}

\DeclareMathOperator{\Gal}{Gal}

\DeclareMathOperator{\id}{id}

\newcommand{\No}{\mathrm N}

\renewcommand{\subset}{\subseteq}
\renewcommand{\supset}{\supseteq}

\newtheoremstyle{theoremdd}
  {10pt}
  {10pt}
  {\itshape}
  {0em}
  {\bfseries}
  {.}
  { }
  {}%
  
\theoremstyle{theoremdd}

\newtheorem{theorem}[equation]{Theorem}
\newtheorem{corollary}[equation]{Corollary}
\newtheorem{proposition}[equation]{Proposition}

\newtheorem{lemma}[equation]{Lemma}
\newtheorem*{lemma*}{Lemma}

\theoremstyle{definition}

\newcommand{\Dk}{{D_k}}

\title{Artin's Conjecture for Abelian Varieties with Frobenius Condition}
\author{Florian Hess\\ University of Oldenburg, Germany \\ \href{florian.hess@uol.de}{florian.hess@uol.de} \and Leonard Tomczak \\ University of Cambridge, UK \\ \href{leonard.tomczak00@gmail.com}{leonard.tomczak00@gmail.com}} 
\date{\today}

\begin{document}

\maketitle

\begin{abstract}
  $A$ be an abelian variety over a number field $K$ of dimension $r$,
  $a_1, \dots, a_g \in A(K)$ and $F/K$ a finite Galois extension.  We
  consider the density of primes $\fp$ of $K$ such that the quotient
  $\bar{A}(k(\fp))/\langle \bar{a}_1,\dots,\bar{a}_g\rangle$ has at
  most $2r-1$ cyclic components and $\fp$ satisfies a Frobenius
  condition with respect to $F/K$, where $\bar{A}$ is the reduction of
  $A$ modulo $\fp$, $k(\fp)$ is the residue class field of $\fp$ and
  $\langle \bar{a}_1,\dots,\bar{a}_g\rangle$ is the subgroup generated
  by the reductions $\bar{a}_1,\dots,\bar{a}_g$. We develop a general
  framework to prove the existence of the density under the
  Generalized Riemann Hypothesis.
\end{abstract}

\section{Introduction}
Let $a\ne-1$ be a non-square integer. Artin's conjecture on primitive roots asserts that the density of the set of primes $p$ for which $a$ is a primitive root exists and is positive. In \cite{HooleyArtinConjecture} Hooley proved the conjecture assuming GRH. In the following we will discuss analogues of this conjecture for elliptic curves and more generally abelian varieties.

Let $E$ be an elliptic curve defined over $\Q$. The immediate analogue of Artin's conjecture would be: Given a point $a\in E(\Q)$ does the set of primes $p$ for which the reduction $\overline{a}$ generates $\overline{E}(\F_p)$ have a density and is it positive? One difference between these questions is that $\F_p^\times$ is always cyclic, and this does not hold for $\overline{E}(\F_p)$. Thus, we can take a step back and ask: Does the set of primes~$p$ for which $\overline{E}(\F_p)$ is cyclic have a density and is it positive? This was analyzed in~\cite{CojocaruMurtyElliptic}. In \cite{AkbalElliptic} the authors also include a congruence condition on the primes to be counted.

Generalizing the question of cyclicity to abelian varieties $A/\Q$ of dimension $r\geq1$ it turns out in \cite{VirdolArtinAV2} that the closest analogue would be to count those primes $p$ for which $\overline{A}(\F_p)$ is the product of at most~$2r-1$ cyclic groups. As with Artin's conjecture we can also consider a finite set of global points $\{a_1,\dots,a_g\}\subset A(\Q)$ and look for those primes for which the $a_1,\dots,a_g$ generate a sufficiently large subgroup in the reduction of $A$. This has been done in \cite{VirdolArtinAV1}.

All of these papers use similar methods for which we develop a general framework. We will then apply this to the concrete situation of abelian varieties and obtain a generalization of the results in \cite{VirdolArtinAV1} and \cite{AkbalElliptic}.

\section{Main Results}\label{section:results}
Let $A$ be an abelian variety defined over a number field $K$ and
$a=\{a_1,\dots,a_g\}$ a set of $K$-rational points. If $k \geq 1$ we let
$L_k=K(A[k],k^{-1}a)$ be the field obtained by adjoining the
coordinates of the $k$-torsion points of $A$ and the coordinates of
the $k$-division points of the $a_i$ to $K$. Then $L_k/K$ is a finite
Galois extension with $\zeta_k \in L_k$, where $\zeta_k$ denotes a
primitive $k$-th root of unity.

Following \cite{VirdolCyclicComp} we say that an abelian group
\textit{has at most $n$ cyclic components} if it can be written as a
product of $n$ cyclic groups, and we say
that $a$ is \textit{primitive-cyclic} modulo some finite prime $\fp$
of $K$ if $\bar{A}(k(\fp))/\langle \bar{a}_1,\dots,\bar{a}_g\rangle$
has at most $2r-1$ cyclic components where $r=\dim A$.

Given a set of finite primes $\PP$ of a number field $K$ and $x\in\R$ we let $$\PP(x):=\#\{\fp\in\PP\mid \No\fp\leq x\}.$$ 

Let $F$ denote a finite Galois extension of $K$. The set of elements $\sigma\in\Gal(L/K)$ fixing a prime of $L$ lying over $\fp$ and inducing the Frobenius automorphism on the corresponding residue field extension will be denoted by $(\fp,L/K)$. 
Let $C_F\subset \Gal(F/K)$ be closed under conjugation.

Our main result is:

\begin{theorem}\label{theorem:av_main_theorem2_extra_condition}
Assuming GRH for the Dedekind zeta functions of the fields $L_kF$, the set $\PP_{C_F}$ of primes $\fp$ of $K$ for which $a$ is primitive-cyclic modulo $\fp$ and which satisfy $(\fp,F/K)\cap C_F\ne\emptyset$, has natural density $$\mathfrak f_{C_F}=\sum_{k = 1}^\infty \mu(k)\delta_{C_F,k} \;\; \text{ where } \;\; \delta_{C_F,k}:=\frac{\#\{\sigma\in\Gal(L_kF/K)\mid \sigma\rvert_F\in C_F,\sigma\rvert_{L_k}=\id_{L_k}\}}{[L_kF:K]},$$ $\mu$ denotes the Möbius function and the series is absolutely convergent. More precisely
$$\PP_{C_F}(x)=\mathfrak f_{C_F}\li(x)+O(x^{5/6}(\log x)^{2/3}).$$
\end{theorem}

We obtain previous results as special cases:
\begin{itemize}
\item[1).] $a=\emptyset, F=K$. This is \cite[Theorem 1.1]{VirdolCyclicComp}.
\item[2).] $F=K=\Q$. This is \cite[Theorem 1.1]{VirdolArtinAV1} with a better error term in most cases.
\item[3).] $A$ elliptic curve, $K=\Q, F=\Q(\zeta_f)$. This is \cite[Theorem 3]{AkbalElliptic}.
\end{itemize}%

Instead of assuming GRH it turns out that the density statement of the theorem, but not the asymptotic formula, can be deduced from the special case $F=K$, i.e. imposing no extra condition on the primes. This is an analogue of \cite[Lemma 3.2]{LenstraArtinConjecture}:

\begin{theorem}\label{theorem:av_unconditional_implication}
Suppose that the set $\PP_{C_K}$ has density $\mathfrak f_{C_K}=\sum_{k=1}^\infty \mu(k)\delta_{C_K,k}$ where $C_K=\{\id_K\}$ and $\delta_{C_K,k}=\frac{1}{[L_k:K]}$. Then $\PP_{C_F}$ has density $\mathfrak f_{C_F}=\sum_{k=1}^\infty \mu(k)\delta_{C_F,k}$. 
\end{theorem}

One remaining question is to analyze when the density $\mathfrak f_{C_F}$ is positive. It is easy to see that in order for $\mathfrak f_{C_F}>0$ to hold it is necessary that $L_q\ne K$ for all primes $q$ and $C_F\ne\emptyset$ (see e.g.\ Proposition \ref{proposition:density_nonnegative}). We believe that for $F=K$ this is also sufficient. Proving this would involve a delicate analysis of the independence of the division fields $L_q$ and suitable subextensions to apply Lemma \ref{lemma:density_bound} as done in \cite[Section 6]{CojocaruMurtyElliptic} in special cases. We will not cover this here.

\section{General Framework}\label{section:general_ideas}
In this and the following section we will write $f\ll g$ or $f=O(g)$ for some functions $f,g$ that may depend on real parameters $k,x$ if there are constants $A,B>0$ such that $|f| \leq B g$ whenever $k,x>A$. Moreover, $f = o(g)$ if for every $\varepsilon > 0$ there is a constant $A>0$ such that $|f|\leq \varepsilon g$ whenever~$k,x>A$. We also write $f(x) \sim g(x)$ if $f(x)/g(x) \to 1$ for $x \to \infty$.

All primes will be finite primes.

\subsection{Density Computation}

The main tool to prove that our sets of primes have certain densities will be an effective version of Chebotarev's density theorem:

\begin{theorem}[{\cite[p.\ 133, Theorem 4]{SerreChebotarev}}]\label{theorem:effective_chebotarev}
Let $L/K$ be a finite Galois extension of number fields and $C\subset\Gal(L/K)$ closed under conjugation. Assuming GRH for the Dedekind zeta function of $L$ we have $$\PP(x)=\frac{\#C}{\#G}\li(x)+O\left(\frac{\#C}{\#G}x^{1/2}(\log d_L+[L:\Q]\log x)\right)$$ where the $O$-constant is absolute and $\PP=\{\fp\subset K\mid\text{$\fp$ unramified in $L$ and $(\fp,L/K)\subset C$}\}$.
\end{theorem}

To bound the discriminant appearing in the error term we shall make use of the following lemma:

\begin{lemma}[{\cite[p.\ 129, Proposition 5]{SerreChebotarev}}]\label{lemma:discriminant_bound}
Let $E/K$ be a Galois extension of degree $n$ of number fields and write $P(E)$ for the set consisting of those rational primes lying under primes of $K$ ramifying in $E$. Then: $$\log d_E\leq n\log d_K+[E:\Q]\left(1-\frac{1}{n}\right)\sum_{p\in P(E)}\log p+[E:\Q]\log n.$$
\end{lemma}

Consider the following setup: Let $K$ be a fixed number field and $(L_q/K)_q$ a family of finite Galois extensions of $K$ where $q$ runs through the set of rational primes. For each $q$ let $C_q\subset\Gal(L_q/K)$ be closed under conjugation. We want to count those primes $\fp$ of $K$ satisfying $(\fp,L_q/K)\cap C_q\ne\emptyset$ for all $q$. We will also want to exclude a small set of primes for which certain characterizations fail, for example in the case of abelian varieties we want to exclude the primes of bad reduction. Thus, let $S$ be a finite set of primes of $K$.\footnote{Instead of taking $S$ to be finite one can also only require $S$ to have density $0$ but this will make some of the following estimates more complicated and one needs to make further assumptions, so we will only deal with the finite case which will suffice for our applications.} For the rest of the paper $k$ denotes a squarefree integer $\geq 1$. As an index to sums $k$ ranges over squarefree integers $\geq 1$ in increasing order. We introduce some further notation:
\begin{align}
\PP_K&:=\{\fp \,|\, \fp \text{ finite prime of } K \},\nonumber\\
\PP_{K,S} &:= \PP_K \backslash S,\nonumber\\
\PP_q&:=\{\fp\in \PP_{K,S} \mid(\fp,L_q/K)\cap C_q\ne\emptyset \},\nonumber\\
\PP &:=\cap_{q} \PP_q = \{\fp\in\PP_{K,S} \mid\text{$(\fp,L_q/K)\cap C_q\ne\emptyset$ for all $q$}\},\nonumber\\
L_k&:= \text{compositum of the $L_q$ with $q | k$},\nonumber\\
G_k&:=\Gal(L_k/K),\nonumber\\
n_k&:=[L_k:\Q],\label{families_extensions_notation}\\
\NN_k&:=\{\fp\in \PP_{K,S} \mid\text{$(\fp,L_q/K)\cap C_q=\emptyset$ for all $q\mid k$}\},\nonumber\\
\Dk &:=\{\sigma\in G_k\mid \sigma\rvert_{L_q}\notin C_q\text{ for all }q\mid k\}, \nonumber\\
\delta_k&:=\frac{\#\Dk}{\#G_k},\nonumber\\
d_k &:= \text{absolute value of the discriminant of $L_k/\Q$},\nonumber\\
r_k &:= \text{product of the rational primes ramifying in $L_k$},\nonumber\\
\mathfrak f & :=\sum_{k}\mu(k)\delta_k.\nonumber
\end{align}

Our goal is to determine conditions such that the density of $\PP$
exists and is equal to $\mathfrak f$. The result is presented in
Theorem \ref{theorem:main_theorem_CM_families} below which we will
derive now.
We also assume right away that the series for $\mathfrak f$ converges
absolutely.

We want to estimate $\PP(x)$ using the inclusion-exclusion principle. To make sense of this we make the following assumption: There is a function $s:\Bbb R_{>0}\to\Bbb R_{>0}$ with the property that for every $x\in\Bbb R_{>0}$ we have: 
\begin{align}\label{sx_definition}
\NN_k(x)\ne\emptyset\Rightarrow k\leq s(x).
\end{align}
In other words, for primes $\fp\notin S$ with norm less than $x$ the condition $(\fp,L_q/K)\cap C_q\ne\emptyset$ is satisfied for some $q\mid k$ whenever $k>s(x)$. We can now write:
\begin{align}\label{PP_formula}
\PP(x) = \sum_{k\leq s(x)}\mu(k)\NN_{k}(x).
\end{align}
We know from the Chebotarev density theorem (Theorem \ref{theorem:effective_chebotarev}) that $\NN_k(x)\sim\delta_k\li(x)$. Thus, we would expect that
\begin{align}\label{PP_goal_prev}\PP(x)\sim \left(\sum_{k\leq s(x)}\mu(k)\delta_k\right)\li(x).
\end{align}
This would imply that
\begin{align}\label{PP_goal}\PP(x)\sim \mathfrak f\li(x)
\end{align}
so that $\PP$ should have density $\mathfrak f$. There is, however, a
problem: In each approximation $\NN_k(x)\approx \delta_k\li(x)$ we get
an error term $e_k(x)$ such that $\NN_k(x) = \delta_k\li(x) +
e_k(x)$. This error term is small enough so that we have
$\NN_k(x)/(\delta_k\li(x)) \to 1$ as $x\to\infty$ for each
$k$ individually. But when summing over all $k$ the sum of these error
terms might become larger than the main term $\mathfrak f\li(x)$ so
that we cannot conclude \eqref{PP_goal_prev}. We now inspect this
error term. Assuming GRH, we know by Theorem
\ref{theorem:effective_chebotarev} that $$e_k(x)=O(\#S)+O(\log
r_k)+O\left(\delta_k\sqrt{x}(\log d_k+n_k\log x)\right).$$ The summand
$O(\#S)$ comes from the set of excluded primes $S$ and the second
summand $O(\log r_k)$ is due to the fact that in Theorem
\ref{theorem:effective_chebotarev} we only count the unramified
primes. Since the number of such primes equals the number of prime
divisors of $r_k$ and $[K:\Q] = O(1)$, it is $\ll \log r_k$.  Hence:
\begin{align*}
\PP(x)&=\sum_{k\leq s(x)}\mu(k)\NN_k(x)
=\Biggl(\underbrace{\sum_{k\leq s(x)}\mu(k)\delta_k}_{=:\mathfrak f_{s(x)}}\Biggr)\li(x)+\sum_{k\leq s(x)}e_k(x)\\
&=\mathfrak f_{s(x)}\li(x)+O(s(x)\#S)+O\Biggl(\sum_{k\leq s(x)}\log r_k\Biggr)+O\Biggl(\sum_{k\leq s(x)}\delta_k\sqrt{x}(\log d_{k}+n_k\log x)\Biggr)\nonumber
\end{align*}
If $\delta_k\ne0$ for all $k$ then $\delta_k n_k \geq 1$,
and the last sum is $\gg \sum_{k\leq s(x)} \sqrt{x}\log x= s(x)\sqrt{x}\log x$. In our applications we can choose $s(x)$ only so that $s(x)\gg \sqrt{x}$. Hence, the this sum is $\gg x\log x$ which is too large. Thus, the bound for the error term coming from the effective Chebotarev theorem is too weak in order to conclude $\PP(x)=\mathfrak f\li(x)+o(\li(x))$. An example by Weinberger (see \cite{WeinbergerCounterexample}) where \eqref{PP_goal} does not hold shows that the sum of the error terms can indeed become too large, whence we need to make further assumptions.

In the treatment of \eqref{PP_formula} we will follow and generalize the method in \cite{CojocaruMurtyElliptic}. The idea will be to require $\NN_k(x)$ to fall off quickly as $k\to\infty$. This will allow us to split up formula \eqref{PP_formula} so that we only need to apply the ``bad'' error term from Chebotarev's theorem for small values of $k$. For this we let $y=y(x)$ to a function between $1$ and $s$ satisfying $y(x)\to\infty$ as $x\to\infty$. We will choose $y$ concretely later in order to obtain the best bounds. We now write:
\begin{align*}
\PP(x)=\underbrace{\sum_{k\leq y(x)}\mu(k)\NN_k(x)}_{\Sigma_{\text{main}}}+\underbrace{\sum_{y(x)<k\leq s(x)}\mu(k)\NN_k(x)}_{\Sigma_{\text{error}}}
\end{align*}
In the same way as before we obtain $$\Sigma_{\text{main}}=\mathfrak f_{y(x)}\li(x)+O(y \#S)+O\Biggl(\sum_{k\leq y(x)}\log r_k\Biggr)+O\Biggl(\underbrace{\sum_{k\leq y(x)}\delta_k\sqrt{x}(\log d_{k}+n_k\log x)}_{=:\tilde{e}(x)}\Biggr)$$ where $\frak f_{y(x)}=\sum_{k\leq y(x)}\mu(k)\delta_k$.
Using Lemma \ref{lemma:discriminant_bound} we get: $$\log d_{k}\leq n_k\left(\log d_K+\log r_k+\log n_k\right)=n_k\log (d_Kr_kn_k).$$
Therefore, observing $n_k \ll \#G_k$, 
\begin{align*}
\tilde{e}(x)&\ll\sum_{k\leq y(x)}\delta_k\#G_k\sqrt{x}(\log (d_Kr_kn_k)+\log x)
\end{align*}
In order to bound this we make the assumption that $r_k$ and $n_k$ grow at most polynomially in $k$, so that $\log (d_Kr_kn_k)\ll \log k$. As $\delta_k\#G_k=\#\Dk\leq \#G_k\leq n_k$ we have $\delta_k\#G_k\ll k^{\gamma}$ for some $\gamma\geq0$. Hence
\begin{align*}
\tilde{e}(x)&\ll\sqrt{x}\Biggl(\sum_{k\leq y(x)}k^\gamma\log k+\log x\sum_{k\leq y(x)}k^\gamma\Biggr)
\ll\sqrt{x}\left(y(x)^{1+\gamma}\log y(x)+y(x)^{1+\gamma}\log x\right).
\end{align*}
It turns out that the value for $y(x)$ we will choose later satisfies $y(x)\leq x$, so we get $$\tilde{e}(x)\ll\sqrt{x}(\log x)y(x)^{1+\gamma}.$$  As both $y(x)\#S$ and $\sum_{k\leq y(x)}\log r_k\ll\sum_{k\leq y(x)}\log k\ll y(x)\log y(x)$ are smaller than this we obtain: 
\begin{align}\label{main_term_estimate}
\Sigma_{\text{main}}=\mathfrak f_y\li(x)+O(\sqrt{x}(\log x)y(x)^{1+\gamma})
\end{align}
We want the error term $| \PP(x) - \mathfrak f_{y(x)}\li(x)|$ for $\PP(x)$ to be $o(\li(x))$, so we have to choose $y(x)$ small enough in~\eqref{main_term_estimate}. But if we take $y(x)$ too small, the sum $\Sigma_{\text{error}}$ will be too big. We make the ansatz $|\Sigma_{\text{error}}|\leq\sum_{y(x)<k\leq s(x)}\NN_k(x)\ll \frac{x^\alpha}{y(x)^\beta}$ for some $\alpha,\beta\geq0$.
Putting this and \eqref{main_term_estimate} together we obtain:
\begin{align*}
\PP(x)=\mathfrak f_y\li(x)+O(\sqrt{x}(\log x)y(x)^{1+\gamma})+O\left(\frac{x^\alpha}{y(x)^{\beta}}\right).
\end{align*}
Since we want both big-$O$ terms to be as small as possible, we pick $y$ as in \cite[p.\ 614]{CojocaruMurtyElliptic} so that we have
$\sqrt{x}(\log x)y(x)^{1+\gamma} =\frac{x^\alpha}{y(x)^{\beta}}$, i.e.\ $$y(x):=\biggl(\frac{x^{\alpha-\frac{1}{2}}}{\log x}\biggr)^{(\beta+\gamma+1)^{-1}}.$$
This yields the error term
$$\sqrt{x}(\log x)y(x)^{1+\gamma}=x^{\frac{1}{2}+(\alpha-\frac{1}{2})\frac{1+\gamma}{\beta+\gamma+1}}(\log x)^{1-\frac{1+\gamma}{\beta+\gamma+1}}=x^{\alpha-\beta\frac{\alpha-\frac{1}{2}}{\beta+\gamma+1}}(\log x)^{\frac{\beta}{\beta+\gamma+1}}=\frac{x^\alpha}{y(x)^{\beta}}.$$
Thus we want $\frac{1}{2}+(\alpha-\frac{1}{2})\frac{1+\gamma}{\beta+\gamma+1}<1$, as in this case the error term is $o(x/\log(x))=o(\li(x))$. We also require $\alpha$ to be $>\frac{1}{2}$ because then $y(x)\to\infty$ as $x\to\infty$ so that $\mathfrak f_{y(x)}\to\mathfrak f$ and hence $$\mathfrak f_y\li(x)=\mathfrak f\li(x)+(\mathfrak f_{y(x)}-\mathfrak f)\li(x)=\mathfrak f\li(x)+o(\li(x)).$$
We let $h(y(x))=\mathfrak f_{y(x)}-\mathfrak f=\sum_{k>y(x)}\mu(k)\delta_k$ and obtain the additional error term $O(h(y(x))\li(x))$. We summarize the result in the following theorem:

\begin{theorem}\label{theorem:main_theorem_CM_families}
  Let $K$ be a number field, for each rational prime $q$ let $L_q$ be a finite Galois extension of $K$ and $C_q\subset\Gal(L_q/K)$ closed under conjugation. Let $S$ be a finite set of primes of $K$. Define all the quantities in \eqref{families_extensions_notation}. Assume that
  \begin{enumerate}
\item \label{thmmain::cond_conv} $\sum_{k} \mu(k) \delta_k$ converges absolutely,
\item \label{thmmain::cond_poly} $r_k,n_k$ grow at most polynomially in $k$,
\item \label{thmmain::cond_sx} there is a function $s:\R_{>0}\to\R_{>0}$ with the property that $$\NN_k(x)\ne\emptyset\implies k\leq s(x)$$ for all squarefree $k$ and $x\in\R_{>0}$,
\item \label{thmmain::cond_abg} there are $\alpha,\beta,\gamma$ such that $\alpha>\frac{1}{2},\beta,\gamma\geq0$, $\frac{1}{2}+(\alpha-\frac{1}{2})\frac{1+\gamma}{\beta+\gamma+1}<1$ and $$\sum_{y(x)<k\leq s(x)}\NN_k(x)\ll \frac{x^\alpha}{y(x)^\beta},\quad c_k\ll k^\gamma,$$
 where $y(x)=\bigl(\frac{x^{\alpha-\frac{1}{2}}}{\log x}\bigr)^{(\beta+\gamma+1)^{-1}}$.
\item GRH holds for all the fields $L_k$.
\end{enumerate}
Then the set $$\PP=\{\fp\in\PP_{K,S} \mid\text{$(\fp,L_q/K)\cap C_q\ne\emptyset$ for all primes $q$}\}$$ has density $\mathfrak f=\sum_{k}\mu(k)\delta_k$. More precisely 
\begin{align*}
\PP(x)=\mathfrak f\li(x)+O\left(h(y(x))\li(x)\right)+O\left(x^{\frac{1}{2}+(\alpha-\frac{1}{2})\frac{1+\gamma}{\beta+\gamma+1}}(\log x)^{1-\frac{1+\gamma}{\beta+\gamma+1}}\right)
\end{align*}
with $h(x)=\sum_{k>x} \mu(k) \delta_k$.
\end{theorem}

The bounds that we obtain for $\sum_{y(x)<k\leq s(x)}\NN_k(x)$ in our applications are not of the above form $x^\alpha/y(x)^\beta$ but rather sums of terms of this form. It turns out that one of these terms will dominate while the others may be neglected. The method of proving the theorem can be easily adapted to show the following corollary.

\begin{corollary} \label{corollary:main_theorem_CM_families_cor}
Let the assumptions be as in the theorem where the fourth point is replaced by:
\begin{enumerate}
\item[\mylabel{thmmain::cond_abgprime}{4'}.] There are $\alpha,\beta,\alpha_1,\beta_1,\dots,\alpha_m,\beta_m,\gamma\geq0$ such that $\alpha>\frac{1}{2}$, $\frac{1}{2}+(\alpha-\frac{1}{2})\frac{1+\gamma}{\beta+\gamma+1}<1$, $\alpha_i-\beta_i\frac{\alpha-\frac{1}{2}}{\gamma+\beta+1}<1$ for all $i=1,\dots,m$ and $$\sum_{y(x)<k\leq s(x)}\NN_k(x)\ll \frac{x^\alpha}{y(x)^\beta}+\sum_{i=1}^m\frac{x^{\alpha_i}}{y(x)^{\beta_i}},\quad c_k\ll k^\gamma,$$
where $y(x)=\bigl(\frac{x^{\alpha-\frac{1}{2}}}{\log x}\bigr)^{(\beta+\gamma+1)^{-1}}$.
\end{enumerate}
Then $\PP$ has density $\mathfrak f$. If $\alpha_i-\beta_i\frac{\alpha-\frac{1}{2}}{\gamma+\beta+1}\leq\frac{1}{2}+(\alpha-\frac{1}{2})\frac{1+\gamma}{\beta+\gamma+1}$, we get the same bound on the error term as in the theorem.
\end{corollary}

\begin{proof}
The other terms in the bound for $\Sigma_{\text{error}}$ yield the additional error $$\sum_{i=1}^mO\Bigl(x^{\alpha_i-\beta_i\frac{\alpha-\frac{1}{2}}{\beta+\gamma+1}}(\log x)^{\frac{\beta}{\beta+\gamma+1}}\Bigr)$$ which is $o(\li(x))$ under the above assumption. If we have $$\alpha_i-\beta_i\frac{\alpha-\frac{1}{2}}{\gamma+\beta+1}\leq\frac{1}{2}+\left(\alpha-\frac{1}{2}\right)\frac{1+\gamma}{\beta+\gamma+1}$$ this error does not exceed the $O$-terms in the expression for $\PP(x)$ in the theorem.
\end{proof}


\subsection{Additional Frobenius Conditions}\label{section:add_frob_cond}

We continue to use the notation~\eqref{families_extensions_notation} and assume
that $\mathfrak f = \sum_k \mu(k) \delta_k$ is an absolutely convergent series.

Suppose that we are given an additional finite Galois extension $F/K$ and a subset $C_F\subset \Gal(F/K)$ closed under conjugation. We now want to count the primes in \begin{align*}
\PP_{C_F}=\{\fp \in \PP_{K,S} \mid\text{$(\fp,F/K)\cap C_F\ne\emptyset$, $(\fp,L_q/K)\cap C_q\ne\emptyset$ for all $q$}\},
\end{align*}
i.e.\ those $\fp$ that satisfy the additional Frobenius condition
$(\fp,F/K)\cap C_F\ne\emptyset$. 

At first sight it seems superfluous to consider an additional Frobenius condition separately: We could simply shift the indices of the family $(L_q/K)_q$ to include $F$ in it and then apply the results from the previous section. However, this situation is still worth studying as it turns out that we can add a single additional Frobenius condition $(\fp,F/K)\cap C_F\ne\emptyset$ \textit{unconditionally} with respect to GRH as long as we know that the conclusion of Theorem \ref{theorem:main_theorem_CM_families} holds. This is the content of Theorem~\ref{lemma:add_frobenius_condition}. It is based on and generalizes \cite[Lemma 3.2]{LenstraArtinConjecture} which only deals with the case~$C_q=\{\id_{L_q}\}$ for all $q$.

\begin{theorem}\label{lemma:add_frobenius_condition}
  Assume that $\PP$ has density $\mathfrak f = \sum_k \mu(k) \delta_k$.
 Then the set $\PP_{C_F}$ has
 density $$\mathfrak f_{C_F} := \sum_{k} \mu(k)\delta_{C_F,k} \;\;
 \text{ where } \;\;
 \delta_{C_F,k}:=\frac{\#\{\sigma\in\Gal(L_kF/K)\mid
   \sigma\rvert_F\in C_F, \sigma\rvert_{L_k}\in \Dk \}}{[L_kF:K]}$$
 and the series is absolutely convergent.
\end{theorem}

For the proof of this theorem we need the following two lemmata. Let $l \geq 1$ be squarefree and define \begin{align*}
  \PP_{C_F,l} & = \{\fp \in \PP_{K,S} \mid\text{$(\fp,F/K)\cap C_F\ne\emptyset$, $(\fp,L_q/K)\cap C_q\ne\emptyset$ for all $q|l$}\} \\
  C_{F, l} & = \{\sigma\in
\Gal(L_lF/K)\mid\sigma\rvert_F\in C_F\text{ and
}\sigma\rvert_{L_{q}}\in C_{q}\text{ for all $q|l$ } \}.\end{align*}

\begin{lemma} \label{lemma::densitycomplementformula} We have
  $$\sum_{k|l} \mu(k) \delta_{C_F, k} = \frac{1}{[L_lF : K]} \# C_{F,l} = d( \PP_{C_F,l} ). $$
\end{lemma}

\begin{proof}
Note that for primes $\fp \in \PP_{K,S}$ that are unramified in $L_lF$ the condition $\fp\in \PP_{C_F, l}$ is equivalent to $(\fp,F/K)\subset C_F$ and $(\fp,L_{q}/K)\subset C_{q}$ for all $q|l$. This is again equivalent to $(\fp,L_lF/K)\subset C_{F,l}$. By the Chebotarev density theorem and the inclusion-exclusion principle
\begin{align*}
d( \PP_{C_F,l} )&=\frac{1}{[L_lF:K]}\# C_{F,l} \\ & = \frac{1}{[L_lF:K]}\sum_{k\mid l}\mu(k)\#\{\sigma\in\Gal(L_lF/K)\mid\sigma\rvert_F\in C_F\text{ and }\sigma\rvert_{L_q}\notin C_q\text{ for all $q\mid k$}\}\\
&=\sum_{k\mid l}\mu(k)\frac{[L_lF:L_kF]}{[L_lF:K]}\#\{\sigma\in\Gal(L_kF/K)\mid\sigma\rvert_F\in C_F\text{ and }\sigma\rvert_{L_q}\notin C_q\text{ for all $q\mid k$}\}\\
&=\sum_{k\mid l}\mu(k)\delta_{C_F,k}. \\[-3.5em] 
\end{align*}
\end{proof}

Let $d$, $d_{\sup}$ and $d_{\inf}$ denote the natural density, upper natural density and lower natural density  of a set of primes.

\begin{lemma}\label{lemma:add_frobenius_condition_helper}
It is always true that $d_{\sup}(\PP_{C_F})\leq \mathfrak f_{C_F}$.
\end{lemma}

\begin{proof}
Let $l$ be the product of the first $n$ prime numbers. Lemma~\ref{lemma::densitycomplementformula}
gives $d( \PP_{C_F, l} ) = \sum_{k|l} \mu(k) \delta_{C_F,k}$. 
Note that $\PP_{C_F}\subset \PP_{C_F, l}$ so that $d_{\sup}(\PP_{C_F})\leq d(\PP_{C_F, l})$. Our claim follows once we let $n\to\infty$.
\end{proof}


\begin{proof}[Proof of Theorem~\ref{lemma:add_frobenius_condition}]
  By assumption $\PP$ has density $\mathfrak f=\sum_{k}
  \mu(k)\delta_k$ and this series is absolutely convergent.  Since $0
  \leq \delta_{C_F,k} \leq \delta_k$ we see that $\mathfrak f_{C_F} =
  \sum_{k} \mu(k)\delta_{C_F,k}$ is also absolutely
  convergent.

  It remains to prove that $d( \PP_{C_F} )$ exists and is equal to $\mathfrak f_{C_F}$.
  We apply Lemma
  \ref{lemma:add_frobenius_condition_helper} to the set
  $\overline{C_F}=\Gal(F/K)\setminus C_F$ to obtain
  \begin{align*}
d_{\sup}(\PP_{\overline{C_F}})\leq \mathfrak f_{\overline{C_F}}.
\end{align*}
Clearly, $\PP_{C_F}\cup \PP_{\overline{C_F}}$ differs from $\PP$ by at most a finite set of primes and similarly $\PP_{C_F}\cap \PP_{\overline{C_F}}$ consists of only finitely many primes. Thus, 
\begin{align*}
\mathfrak f=d(\PP)=d(\PP_{C_F}\cup \PP_{\overline{C_F}})&=d_{\inf}(\PP_{C_F}\cup \PP_{\overline{C_F}})\\
&\leq d_{\inf}(\PP_{C_F})+d_{\sup}(\PP_{\overline{C_F}})\\
&\leq d_{\inf}(\PP_{C_F})+\mathfrak f_{\overline{C_F}},
\end{align*}
where the first inequality uses the general fact that $\liminf (a_n +
b_n) \leq \liminf( a_n) + \limsup( b_n )$ for bounded sequences $(a_n)_n, (b_n)_n$
in $\mathbb R$.  It is easy to see that
$\delta_{k}=\delta_{C_F,k}+\delta_{\overline{C_F},k}$ so that
$\mathfrak f=\mathfrak f_{C_F} + \mathfrak f_{\overline{C_F}}$. The above inequalities and  Lemma
\ref{lemma:add_frobenius_condition_helper} yield $$\mathfrak f_{C_F} = \mathfrak f - \mathfrak f_{\overline{C_F}} \leq d_{\inf}(\PP_{C_F}) \leq d_{\sup}(\PP_{C_F}) \leq \mathfrak f_{C_F},$$ whence $d_{\inf}(\PP_{C_F})$ exists and is equal to $\mathfrak f_{C_F}$.
\end{proof}

We can also extend Theorem \ref{theorem:main_theorem_CM_families}
directly to include the additional Frobenius condition. Note that in
the following theorem $F$ only occurs in the
assumption of GRH for the fields $L_k F$ and not in any other
conditions that need verification in order to apply the theorem.

\begin{theorem}\label{theorem:main_theorem_CM_families_extra_condition}
Assume that GRH holds for the fields $L_kF$ and let otherwise the assumptions be as in Theorem \ref{theorem:main_theorem_CM_families} or in Corollary \ref{corollary:main_theorem_CM_families_cor}. Then $\PP_{C_F}$ has density $\frak f_{C_F}$ and the error term is the same as in the theorem.
\end{theorem}

\begin{proof}
The same proof as the one of Theorem~\ref{theorem:main_theorem_CM_families} works: We replace $\PP$ by $\PP_{C_F}$ and $\NN_k$ by $$\NN_{k,C_F}=\{\fp\in\PP_{K,S}\mid (\fp,F/K)\cap C_F\ne\emptyset \text{ and }(\fp,L_q/K)\cap C_q=\emptyset \text{ for all }q\mid k\}.$$ We then write $$\PP_{C_F}(x)=\sum_{k\leq y(x)}\mu(k)\NN_{k,C_F}(x)+\sum_{y(x)<k\leq s(x)}\mu(k)\NN_{k,C_F}(x)$$ and proceed as before. 
\end{proof}

\subsection{Positivity of the Density}

We continue to use the notation~\eqref{families_extensions_notation} and assume
that $\mathfrak f = \sum_k \mu(k) \delta_k$ is an absolutely convergent series.

Even if we knew that $\PP$ has density $$\mathfrak f = \sum_{k} \mu(k)\delta_k,$$ it is not obvious at all whether or not it is positive.
In this section we want to deal with $\mathfrak f$ independently of Theorem~\ref{theorem:main_theorem_CM_families} and its assumptions. We will use only combinatorical arguments and derive some simple criteria for~$\mathfrak f > 0$.


\begin{proposition}\label{proposition:density_nonnegative}
We have $\mathfrak f \geq 0$, and if $C_q = \emptyset$ for some prime $q$ then $\mathfrak f=0$.
\end{proposition}

\begin{proof}
  Let $l$ be the product of the first $n$ prime numbers. Then
  $\sum_{k|l}\mu(k)\delta_k \geq 0$ by
  Lemma~\ref{lemma::densitycomplementformula} applied with $F = K$ and
  $C_F = \{ \id_K \}$. If $C_q=\emptyset$ for some $q|l$ then $C_{F,l}
  = \emptyset$ and $\sum_{k| l}\mu(k)\delta_k=0$.
  Taking the limit $n \to \infty$ proves the assertion.
\end{proof}

Under additional assumptions we get the following converse to Proposition~\ref{proposition:density_nonnegative}.

\begin{proposition} \label{prop::densitynonneg}
Suppose that the family $(L_q/K)_q$ is independent in the sense that $L_k$
and $L_{k'}$ are linearly disjoint over $K$ whenever $k$ and $k'$ are
coprime. Then $$\mathfrak f=\prod_{q}(1-\delta_q)$$ In particular,
 if $C_q\ne\emptyset$ for all primes $q$ then $\mathfrak f>0$.
\end{proposition}

\begin{proof}
  Let $k,k'$ be coprime. By assumption the canonical map $G_{kk'}\to
  G_k\times G_{k'}$ is a bijection. It restricts to a bijection
  between $G_{kk'}\setminus C_{kk'}$ and $(G_{k}\setminus
  C_{k})\times(G_{k'}\setminus C_{k'})$, in other words the map
  $g:k\mapsto \mu(k)\delta_k=\mu(k)\frac{\#\Dk}{\#G_k}$ is
  multiplicative with $g(1) = 1$. Hence $$\mathfrak f=\sum_{k} g(k) =
  \prod_{q}(1+g(q))=\prod_q(1-\delta_q),$$ and the latter product
  converges to a non zero value if and only if $\delta_q < 1$ for all
  $q$ and $\sum_{q} \delta_q$ converges. Now $C_q \neq \emptyset$
  implies $\delta_q < 1$ and the convergence of the series $\sum_q \delta_q$ follows from
  the absolute convergence of $\sum_k \mu(k) \delta_k$.
\end{proof}

In the setting of abelian varieties considered in Section \ref{section:proofs} it will usually not be the case that the fields $L_k$ are completely independent. But in some cases, as for example~\cite[p.~621]{CojocaruMurtyElliptic},  this family of fields will contain other fields which form an independent family, and the problem can be reduced to these other fields. This can be done 
using the following lemma which is a generalization of \cite[Lemma 6.1]{CojocaruMurtyElliptic}.

\begin{lemma}\label{lemma:density_bound}
  Suppose we have another family $(L'_q/K)_{q\in U}$ of finite Galois extensions indexed by some subset $U$ of the set of rational primes and $C'_q\subset \Gal(L'_q/K)$ closed under conjugation. Define the fields $L'_k$ and the numbers $\delta'_k,\mathfrak f'$ as before for this family of fields for those squarefree $k$ whose prime factors are in $U$. Assume that
  \begin{enumerate}
  \item Every $L'_{q'}$ is contained in some $L_q$ and conversely every $L_q$ contains some $L'_{q'}$.
\item For each prime $q$ there are only finitely many $q'$ such that $L'_{q'}\subset L_q$.
\item If $L'_{q'}\subset L_q$, we have $\pi^{-1}(C'_{q'})\subset C_q$ where $\pi:\Gal(L_q/K)\to \Gal(L'_{q'}/K)$ is the restriction.
\item $\sum_k \mu(k) \delta'_k$ converges absolutely.
\end{enumerate}
Then $\mathfrak f\ge\mathfrak f'$.
\end{lemma}

\begin{proof}
 Let $T$ be a finite set of primes and $T'$ be the set
of primes $q'\in U$ such that $L'_{q'}\subset L_q$ for some $q\in
T$. By assumption~2 the set $T'$ is finite. Since every $L'_{q'}$ is
contained in some $L_q$ by assumption~1 the union of the sets $T'$ is
the set $U$ whenever the union of all $T$ is the set of all
primes. Let $l$ and $l'$ denote
the product of the primes in $T$ and in $T'$ respectively.
Thus, as in Proposition~\ref{proposition:density_nonnegative} and using
assumption~4, it suffices to prove
that \begin{align} \label{eq::density_bound0}
  \sum_{k|l}\mu(k)\delta_k\geq\sum_{k|l'}\mu(k)\delta'_k.
\end{align}
By Lemma~\ref{lemma::densitycomplementformula} applied with $F = K$ and
  $C_F = \{ \id_K \}$,  \eqref{eq::density_bound0} is
equivalent to
\begin{align}
\#\{\sigma\in \Gal (L_l/K)\mid & \; \sigma\rvert_{L_q}\in C_q\text{ for all } q|l\}\geq\nonumber\\ & [L_l:L'_{l'}] \, \#\{\sigma\in \Gal (L'_{l'}/K)\mid \sigma\rvert_{L'_{q'}}\in C'_{q'}\text{ for all } q'| l'\}\label{covering_ineq_to_prove} 
\end{align}
Let $\sigma\in\Gal (L'_{l'}/K)$ such that $\sigma\rvert_{L'_{q'}}\in C'_{q'}$ for all $q'\mid l'$ and let $\tau\in\Gal (L_l/K)$ be an extension of $\sigma$ to $L_l$. For every $q\in T$ there is $q'\in T'$ such that $L'_{q'}\subset L_q$ by assumption~1. Let $\pi:\Gal(L_q/K)\to \Gal(L'_{q'}/K)$ be the restriction. We have $\tau\rvert_{L_q}\in \pi^{-1}(C'_{q'})\subset C_q$ since $\tau\rvert_{L'_{q'}}=\sigma\rvert_{L'_{q'}}\in C'_{q'}$ and by assumption 3. Thus, $\tau$ is contained in the set on the left side of \eqref{covering_ineq_to_prove} and since every $\sigma$ admits exactly $[L_l:L'_{l'}]$ extensions to $L_l$, the inequality follows.
\end{proof}

The proof in \cite{CojocaruMurtyElliptic} is slightly different as they use Chebotarev's density theorem to prove~\eqref{eq::density_bound0} while we establish this inequality by purely combinatorical arguments.

\section{Proofs of the Main Results}\label{section:proofs}

Let $A$ be an abelian variety defined over a number field $K$ and $N$ its conductor. Let $\{a_1,\dots,a_g\}$ be a finite set of points in $A(K)$. As before for a rational prime $q$ prime let $L_q:=K(A[q], q^{-1}a)$. The following lemma gives a characterization of the primes we are interested in that will then allow us to use the results from Section \ref{section:general_ideas}.

\begin{lemma} \label{lemma::VirdolArtinAV2}
Let $\fp$ be a prime of $K$ such that $(\fp,N)=1$. Then for a rational prime $q\nmid \fp$ the group $\bar{A}(k(\fp))/\langle \bar{a}_1,\dots,\bar{a}_g\rangle$ contains a subgroup isomorphic to $(\Z/q\Z)^{2r}$ if and only if $\fp$ splits completely in $L_q$. Consequently, $a$ is primitive-cyclic modulo $\fp$ if and only if $\fp$ does not split in $L_q$ for all~$q\nmid\fp$.
\end{lemma}
\begin{proof}
  We refer to \cite[Lemma 2.2]{VirdolArtinAV2}. There the $a_i$ are
  assumed to be independent, but this is not used in the proof. Hence
  we state the lemma in its more general form for our more general
  context.
\end{proof}


Note that the field $K$ contains only a finite number of cyclotomic fields. Let $l$ be the least common multiple of those integers $m$ for which $\zeta_m\in K$. In the rest of this section we shall only consider primes $\fp$ not dividing $lN$. Using the notation from Section \ref{section:general_ideas} we then have $S=\{\fp\in \PP_K \mid \fp\mid lN\}$, set $C_q=G_q\setminus\{\id_{L_q}\}$ and the quantities in
\eqref{families_extensions_notation}. Note that the two definitions of $L_k$ either as the composite field of the $L_q$ for $q\mid K$ or as $K(A[k],k^{-1}a)$ coincide.

Let $\fp$ a prime of $K$ not dividing $lN$ and $p$ the rational prime lying under $\fp$. Then $K\subsetneq K(\zeta_p)\subset L_p$ and hence $\fp$ is ramified and in particular does not split completely in $L_p$. Thus, up to the finite number of primes dividing $l$ the set of primes $\fp$ such that $a$ is primitive-cyclic mod $\fp$ coincides with $\PP$ according to Lemma~\ref{lemma::VirdolArtinAV2}. Having translated our problem into the language of the general framework in Section~\ref{section:general_ideas}, we immediately deduce Theorem \ref{theorem:av_unconditional_implication} from Theorem~\ref{lemma:add_frobenius_condition}.


We now want to verify the conditions, apart from GRH obviously, in Theorem \ref{theorem:main_theorem_CM_families} following~\cite{VirdolCyclicComp}. Let $\fp\in\NN_k(x)$, i.e.\ $\fp$ splits completely in $L_k$, $\No\fp\leq x$ and $\fp\nmid lN$. Then $\fp$ also splits completely in all $L_q$ for primes $q$ dividing $k$. Applying the lemma to these $q$ we obtain that $\bar{A}(k(\fp))/\langle \bar{a}_1,\dots,\bar{a}_g\rangle$ contains a subgroup isomorphic to $(\Z/q\Z)^{2r}$, in particular $q^{2r}\mid\#\overline{A}(k(\fp))$ and hence $k^{2r}\mid\#\overline{A}(k(\fp))$. If $\alpha_1,\dots,\alpha_{2r}$ denote the zeros of the characteristic polynomial of the Frobenius on $\overline A$, using $|\alpha_i| = (\No \fp)^{1/2}$ we get the following bound: $\#\overline{A}(k(\fp))=\prod_{i=1}^{2r}(1-\alpha_i)\leq \prod_{i=1}^{2r}2|\alpha_i|=2^{2r}\No\fp^r$. Hence $k\leq 2\No\fp^{1/2}\leq 2\sqrt{x}$. This shows that condition~\ref{thmmain::cond_sx} of Theorem~\ref{theorem:main_theorem_CM_families} is satisfied with $s(x)=2\sqrt{x}$. 

By \cite[Proposition C.1.5, Remark F.3.3]{HindrySilvermanDiophantineGeometry} and \cite[Theorem 1]{SerreTateGoodReduction} a prime $\fp$ of $K$ ramifying in $L_k$ has to divide $kN$. It follows that $r_k\mid kd_K N$.  For each $a_i$ the Galois group of $K(A[k],k^{-1}a_i)/K(A[k])$ embeds into $(\Z/k\Z)^{2r}$, see \cite[Lemma C.1.3]{HindrySilvermanDiophantineGeometry} and $\Gal(K(A[k])/K)\hookrightarrow\GL_{2r}(\Z/k\Z)$. From this we can bound $$n_k=[L_k:\Q] \leq [K:\Q]\prod_{i=1}^g \Bigl( [K(A[k],k^{-1}a_i):K(A[k])] \cdot [K(A[k]):K] \Bigr)\leq k^{2rg+4r^2}[K:\Q].$$ Thus, $r_k$ and $n_k$ grow at most polynomially. This shows that condition~\ref{thmmain::cond_poly} of Theorem~\ref{theorem:main_theorem_CM_families} is satisfied. 

As $\delta_k\#G_k=\# \Dk =\#\{\id_{L_k}\}=1$ we can take
$\gamma=0$. Estimating $\NN_k(x)$ will be slightly more difficult:

\begin{lemma}
  $\NN_k(x)$ satisfies $$\NN_k(x)\ll \frac{x^{3/2}}{k^3}+\frac{x}{k^2}$$ uniformly in $k$ and hence$$\sum_{y(x)\leq k\leq s(x)}\NN_k(x)\ll \frac{x^{3/2}}{y(x)^2}+\frac{x}{y(x)}.$$ for all functions $y$ with $y(x)\geq1$.
\end{lemma}

\begin{proof}
  In the first half of the proof we adapt \cite[pp.\ 431, 432]{VirdolCyclicComp}\footnote{The proof is similar as in \cite{VirdolCyclicComp} but there the author does not distinguish the cases $a_{\fp}=2r$ and $a_{\fp}\ne2r$ which is necessary and also done in \cite{CojocaruMurtyElliptic}, hence we present the complete argument here.} to our setting and write
\begin{align*}
  S_a(k)&=\{\fp\in\NN_k\mid a_\fp=a,\;\No\fp\leq x\}\\&=\{\fp\mid a_\fp=a,\;\fp\nmid lN,\;\fp\text{ splits completely in $L_k$, $\No\fp\leq x$}\}
\end{align*}
where $a_\fp$ is the trace of the Frobenius endomorphism of $\overline{A}$. Note that $|a_\fp|\leq 2r\No\fp^{1/2}$ by the Riemann hypothesis. Then $\{\fp\in \NN_k\mid \No\fp\leq x\}$ is the disjoint union $\bigcup_{\abs{a}\leq 2r\sqrt{x}} S_a(k)$. Thus
\begin{align*}
\NN_k(x)=\sum_{\abs{a}\leq 2r\sqrt{x}}\#S_a(k).
\end{align*}
If $\fp\in S_a(k)$ we obtain several divisibility relations which lead to upper bounds for $\# S_a(k)$. Firstly, we know, as $\fp$ splits completely in $L_k\supset K(\zeta_k)$, that the field $k(\fp)$ contains a primitive $k$-th root of unity, hence $k\mid \#k(\fp)-1=\No\fp-1$. Secondly, by \cite[Lemma 2.6]{VirdolCyclicComp} we have $k^2\mid r\No\fp+r-a_\fp$. Hence \begin{align} \label{eqsak1} S_a(k) & \subset\{\fp\mid \No\fp\leq x, \, a_\fp = a, \,k^2\mid r\No\fp+r-a_\fp\}.\end{align} Note that $|r\No\fp+r-a_\fp|\leq rx+r+2r\sqrt{x}=r(\sqrt{x}+1)^2\leq 4rx$. Since $r\No\fp+r-a_\fp \neq 0$ for all but possibly finitely many $\fp$ there are at most $\frac{4rx}{k^2}$ multiples of $k^2$ less than or equal to $4rx$ and for each such multiple there are at most $[K:\Q]$ primes of $K$ with corresponding norm. Therefore \eqref{eqsak1} implies $$\#S_a(k)\ll\frac{x}{k^2}.$$
We now adapt the strategy of
\cite{CojocaruMurtyElliptic}. We split up the sum
over the $a$ noting that the divisibility relations above also imply
that~$k\mid-(r\No\fp+r-a_\fp)+r(\No\fp-1)= a_\fp-2r$. Observing $| a_\fp - 2r | \leq 2r (\sqrt{x}+1)$ we get
\begin{align*}
\sum_{\abs{a}\leq 2r\sqrt{x}}\#S_a(k)&\leq \# S_{2r}(k) + \sum_{\substack{a\in\Z\setminus\{2r\}\\\abs{a}\leq 2r\sqrt{x}\\k\mid a-2r}} \# S_a(k)  \ll\frac{x}{k^2}+\sum_{\substack{a\in\Z\setminus\{2r\}\\\abs{a}\leq 2r\sqrt{x}\\k\mid a-2r}}\frac{x}{k^2} \ll\frac{x}{k^2}+\frac{\sqrt{x}}{k}\frac{x}{k^2}=\frac{x^{3/2}}{k^3}+\frac{x}{k^2}.\\[-3em]
\end{align*}
\end{proof}

The lemma shows that the condition~\ref{thmmain::cond_abgprime} of
Corollary~\ref{corollary:main_theorem_CM_families_cor} is satisfied with
the constants $\alpha=\frac{3}{2}$, $\beta=2$, $\alpha_1=1$, $\beta_1=1$.
Then \begin{align} \label{eqfinal1} x^{\frac{1}{2}+(\alpha-\frac{1}{2})\frac{1+\gamma}{\beta+\gamma+1}}
(\log x)^{1-\frac{1+\gamma}{\beta+\gamma+1}} = x^{5/6} (\log x)^{2/3}.\end{align}
Using the notation from Section \ref{section:general_ideas} we set $h(x):=\sum_{k>x}\frac{\mu(k)}{\#G_k}$. By \cite[Lemma 2.2]{VirdolCyclicComp} we have $\#G_k\gg k^{3/2}$, thus $|h(y(x))|\ll y(x)^{-1/2}$. This shows that condition~\ref{thmmain::cond_conv} of Theorem~\ref{theorem:main_theorem_CM_families} is satisfied. Also \begin{align} \label{eqfinal2} |h(y(x))| \li(x) \ll y(x)^{-1/2} \li(x) \sim  \left( \bigl(\frac{x^{\alpha-\frac{1}{2}}}{\log x}\bigr)^{(\beta+\gamma+1)^{-1}} \right)^{-1/2} \frac{x}{\log(x)} = \left( \frac{x}{\log x} \right)^{5/6},\end{align} and \eqref{eqfinal2} is dominated by \eqref{eqfinal1}.
Combining all this allows us to apply Theorem~\ref{theorem:main_theorem_CM_families_extra_condition}
which proves Theorem \ref{theorem:av_main_theorem2_extra_condition}.

\printbibliography

\end{document}